\documentclass[11pt]{article}
\pdfoutput=1 

\usepackage{hyperref}
\usepackage{url,float}
\usepackage{graphicx}
\usepackage{amsmath}
\usepackage{amsfonts}
\usepackage{amssymb}
\usepackage{latexsym}
\usepackage{color}
\usepackage{xcolor}
\usepackage[shortlabels]{enumitem}

\usepackage{hyperref}
\hypersetup{
    colorlinks=true,
    linkcolor=blue,
    filecolor=magenta,      
    urlcolor=cyan,
    }

\newcommand{\hide}[1]{}

\newcommand{\ABox}{
\raisebox{3pt}{\framebox[6pt]{\rule{6pt}{0pt}}}
}
\newenvironment{proof}{{\bf Proof:}}{\hfill\ABox\medskip}

\newtheorem{theorem}{{\bf Theorem}}
\newtheorem{corollary}[theorem]{Corollary}

\newtheorem{lemma}[theorem]{Lemma}

\newtheorem{definition}[theorem]{Definition}
\newtheorem{obs}[theorem]{Observation}
\newtheorem{claim}[theorem]{Claim}

\newcommand{\lemlab}[1]{\label{lemma:#1}}
\newcommand{\thmlab}[1]{\label{thm:#1}}

\newcommand{\figlab}[1]{\label{fig:#1}}
\newcommand{\seclab}[1]{\label{sec:#1}}

\newcommand{\lemref}[1]{\ref{lemma:#1}}
\newcommand{\thmref}[1]{\ref{thm:#1}}

\newcommand{\secref}[1]{\ref{sec:#1}}

\newcommand{\figref}[1]{\ref{fig:#1}}

\def\a{{\alpha}}
\def\b{{\beta}}

\def\D{{\mathcal D}}
\def\o{{\omega}}
\def\d{{\delta}}

\def\R{{\mathbb{R}}}
\def\S{{\mathbb{S}}}

\def\defn#1{\textit{\textbf{\boldmath #1}}}

\newcommand{\squeezelist}{\setlength{\itemsep}{0pt}}

\newcommand{\JOR}[1]{{\color{red} JOR: [#1]}}
\newcommand{\jor}[1]{{\color{cyan} #1}}
\newcommand{\Anna}[1]{{\color{red} Anna: [#1]}}
\newcommand{\anna}[1]{{\color{blue} #1}}

{\makeatletter
\gdef\@fnsymbol#1{\ensuremath{\ifcase#1\or *\or \dagger\or \ddagger\or
   \mathsection\or \mathparagraph\or \|\or **\or \dagger\dagger\or
   \ddagger\ddagger\or \mathsection\mathsection\or
   \mathparagraph\mathparagraph\or \|\|\or ***\or \dagger\dagger\dagger\or
   \ddagger\ddagger\ddagger\or \mathsection\mathsection\mathsection\or
   \mathparagraph\mathparagraph\mathparagraph\or \|\|\|\else\@ctrerr\fi}}}

\title{%
Deltahedral Domes over Equiangular Polygons\footnote{%
Full version of abstract presented at EuroCG~\cite{DomesEuroCG}.
J.T. was supported by the Center for Foundations of Modern Computer Science (Charles Univ. project UNCE/SCI/004) and by project PRIMUS/24/SCI/012 from Charles Univ.}
} 

\author{%
MIT CompGeom Group\thanks{%
Artificial first author to highlight that the other authors
(in alphabetical order) worked as an equal group.
Please include all authors (including this one) in your bibliography,
and refer to the authors as ``MIT CompGeom Group'' (without ``et al.'').}
\and
Hugo A. Akitaya\thanks{%
U. Mass. Lowell, \texttt{hugo\_akitaya@uml.edu}}
\and
Erik D. Demaine\thanks{%
MIT, \texttt{edemaine@mit.edu}}
\and
Adam Hesterberg\thanks{%
Harvard U., \texttt{achesterberg@gmail.com}}
\and
Anna Lubiw\thanks{%
U. Waterloo, \texttt{alubiw@uwaterloo.ca}}
\and
Jayson Lynch\thanks{%
MIT, \texttt{jaysonl@mit.edu}}
\and
Joseph O'Rourke\thanks{%
Smith College, \texttt{jorourke@smith.edu}}
\and
Frederick Stock\thanks{%
U. Mass. Lowell, \texttt{fbs9594@rit.edu}} 
\and
Josef Tkadlec\thanks{%
Charles U.
\texttt{jtkadlec@ist.ac.at}}
}

\begin{document}

\date{}
\maketitle

\begin{abstract}
A \emph{polyiamond} is a polygon composed of unit equilateral triangles, and
a \emph{generalized deltahedron}
is a convex polyhedron 
whose every face is a convex polyiamond.
We study a variant where one face may be an exception.
For a convex polygon $P$, 
if there is a convex polyhedron that has $P$ as one face and all the other faces are convex polyiamonds, then we say that 
$P$ \emph{can be domed}.
Our main result is a complete characterization of 
which equiangular $n$-gons can be domed:
only if $n \in \{3,4,5,6,8,10,12\}$, and only with some 
conditions on the integer edge lengths.
\end{abstract}

\section{Introduction}
\seclab{Introduction}
In the study of what can be built with equilateral triangles, the most well-known result is that there are exactly eight convex \defn{deltahedra}---polyhedra where every face is an equilateral triangle---with $n = 4,5,6,7,8,9,10,12$ vertices.
See references in \cite{bezdek2007proof} or
Wikipedia.\footnote{\url{https://en.wikipedia.org/wiki/Deltahedron}}
What if coplanar triangles are allowed?
In the plane, the polygons built of equilateral triangles are the \defn{polyiamonds}. Convex polyiamonds have 3, 4, 5, or 6 vertices.
The convex polyhedra with polyiamond faces are the 
``non-strictly convex deltahedra'', or \defn{generalized deltahedra}, following the nomenclature of Bezdek~\cite{bezdek2007proof}.
See the above cited Wikipedia article
for 
examples. 
There are an infinite number of generalized deltahedra, though  the number of combinatorial types is finite since they have at most $12$ vertices.
There is no published characterization, though a forthcoming one is mentioned in~\cite{bezdek2007proof}.

Our goal (only partially achieved) 
is to characterize when a convex polygon can be ``domed'' with a convex surface composed of equilateral triangles.  
For a convex polygon $P$, 
if there is a convex polyhedron that has $P$ as one face and all the other faces are convex polyiamonds, then we say that $P$ \defn{can be deltahedrally domed}, or just \defn{domed} for short.  
Here the \defn{deltahedral dome} (\defn{dome} for short), denoted by $\D$, is the part of the polyhedron excluding face $P$, and $P$ is called the \defn{base} of the dome.  Note that $P$ itself may or may not be a polyiamond.


We assume that all the equilateral triangles have unit edge length,
so $P$ must be an integer polygon (with integer side lengths).
Here is a simple example:
\begin{lemma}
\lemlab{Rect}
Every integer rectangle can be domed.
\end{lemma}
\begin{proof}
If the base $P$ is an $a \times b$ rectangle, $a \ge b$,
then one can construct a ``roof" dome
whose top ridge has length $a-b$, and whose opposite
faces are trapezoids and triangles.
See Figs.~\figref{RectInt} and~\figref{Rect_3x1_roof}.
\end{proof}

\begin{figure}[htbp]
\centering
\includegraphics[width=0.55\textwidth]{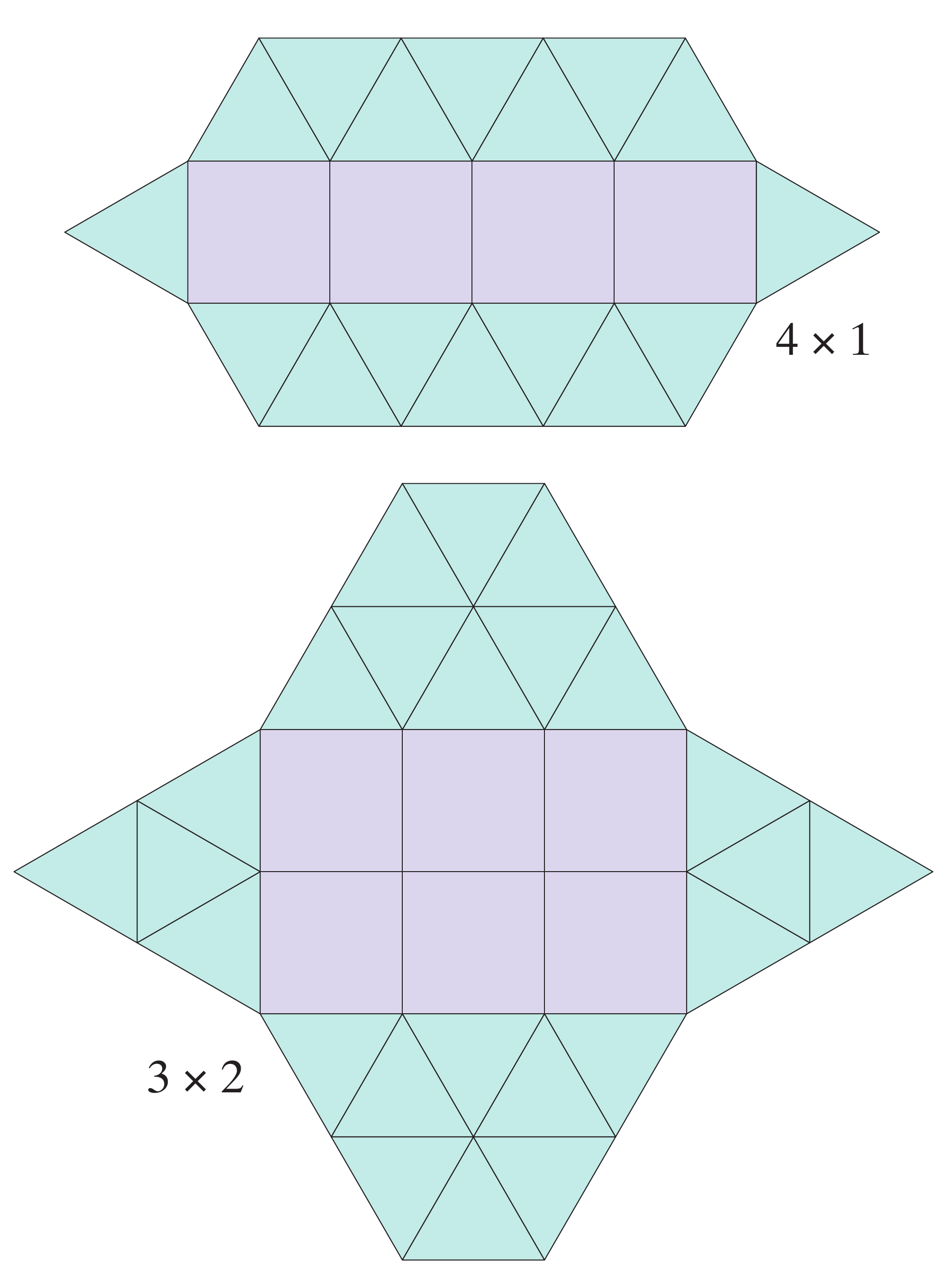}
\caption{Integer rectangle $a \times b$:
dome faces are  trapezoids and triangles.
}
\figlab{RectInt}
\end{figure}
\begin{figure}[htbp]
\centering
\includegraphics[width=0.75\textwidth]{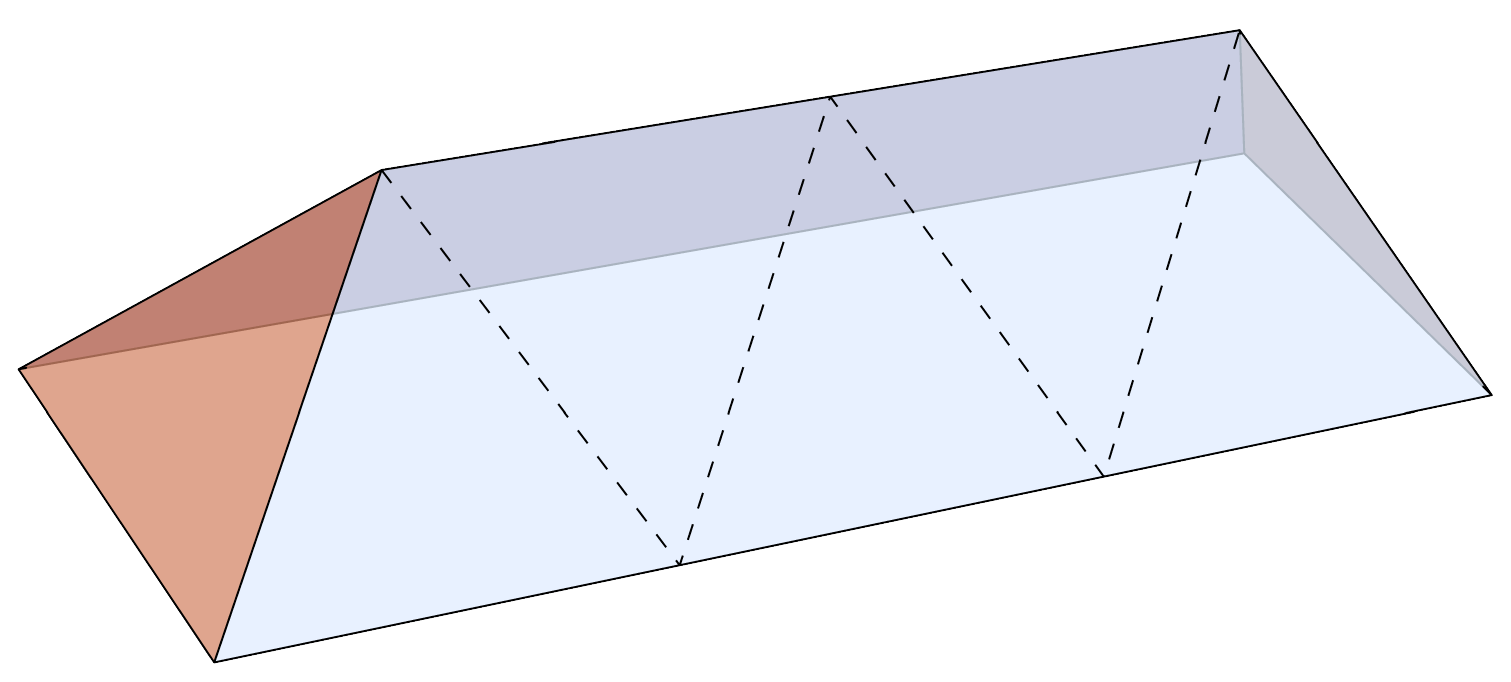}
\caption{``Roof'' dome over a $3 \times 1$ rectangle.}
\figlab{Rect_3x1_roof}
\end{figure}

\subsection{Main Theorem}
\begin{theorem}
\thmlab{EquiAng}
\label{thm:main}
(a)~The only equiangular 
polygons
that can be domed have $n$ vertices, where $n \in \{3,4,5,6,8,10,12\}$;
for each such $n$, any regular integer $n$-gon  
can be domed.
(b)~Moreover, 
for $n=3,4,5,6$ every 
equiangular integer polygon can be domed, and for $n=8,10,12$, an 
equiangular integer $n$-gon can be domed iff the 
lengths of the odd-numbered edges 
are equal and the even edge lengths are those of an equiangular $\frac{n}{2}$-gon.
\end{theorem}

For small $n$, edge-length conditions for an equiangular integer polygon are known~\cite{ball2002equiangular}: 
for $n=4$ these are rectangles; for $n=5$ there is only the regular pentagon; and for $n=6$ the edge lengths must be integers $a,b,c,a',b',c'$ with $a-a' = b'-b = c-c'$ (a 6-sided polyiamond);
see ahead to Fig.~\figref{Polyiamond6}.

Part~(a) of Theorem~\ref{thm:main} is proved in Sections~\ref{sec:RegPoly} and~\ref{sec:part-a}.
Part (b) is proved in Section~\ref{sec:main-part-b}.
We have established several results beyond the main theorem (for example
that there is no domeable polygon with $25$ or more vertices)---see Section~\ref{sec:further-results}.


\subsection{Glazyrin and Pak}
The source of our work derives from a paper by
Glazyrin and Pak: ``Domes over Curves''~\cite{glazyrin2022domes},
which answers a question posed by Richard Kenyon in 2005.\footnote{%
\url{https://gauss.math.yale.edu/~rwk25/openprobs/}.}
In~\cite{glazyrin2022domes}, a ``curve'' $P$ is a closed polygonal chain in $\R^3$,
and a dome is 
a PL-surface composed of unit equilateral triangles. 
They say that $P$ can be \emph{spanned}
if there is such a surface
whose boundary is $\partial P$.
We note the following two differences with our definitions:
\begin{enumerate}[(1)]
\item Our $P$ is a 2D convex polygon; theirs is a 3D possibly self-intersecting polygonal chain.
\item Our dome $\D$ is embedded (non-self-intersecting) and convex.
Their PL-surface is (in general) nonconvex, immersed, and self-intersecting.
\end{enumerate}
Under their conditions, they show that 
certain nonplanar rhombi cannot be spanned, which answers Kenyon's question in the negative.\footnote{%
Recent work~\cite{anan2024moduli}
extends the Glazyrin-Pak negative result to
show that ``generic'' integer polygons cannot be spanned. See~\cite{CobordismDomes} for a further generalization.
}
More interesting for our purposes, they prove that every planar regular polygon can be spanned
(their Theorem~1.4).
In contrast, our Theorem~\thmref{EquiAng} says that the regular $7$-, $9$-, and $11$-gons cannot
be domed, nor can any regular $n$-gon for $n>12$.
And here ``regular'' can be 
strengthened to ``equiangular.''
Compared to their results,
our conditions constrain the geometry and limit what can be domed.


\section{Domed Regular Polygons}
\seclab{RegPoly}
We prove one part of Theorem~\ref{thm:main}(a)
by exhibiting domes over regular integer $n$-gons, for $n \in \{3,4,5,6,8,10,12\}$.
We will use $\bar{P}_n$ to denote a regular integer $n$-gon.

\begin{itemize}
\item $3,4,5$ : $\bar{P}_n$ for $n=3,4,5$ can each be domed by a pyramid:
Fig.~\figref{Pyramids345}.
\item $6$ : Hexagonal antiprism: Fig.~\figref{SliceFigs}(a).
\item $8$ : A slice through a
gyroelongated
square diprism: Fig.~\figref{SliceFigs}(b).
\item $10$ : A slice through an icosahedron: Fig.~\figref{SliceFigs}(c).
\item $12$ : A slice through a hexagonal antiprism: Fig.~\figref{SliceFigs}(d).
\end{itemize}

A few remarks.
The pyramid pattern for $n=3,4,5$ cannot be extended to $\bar{P}_6$,
for that would result in a doubly-covered hexagon, not a
dome by our definition.
For $n=8,10,12$, we show $\bar{P}_n$ as a slice of a convex polyhedron,
with the dome the upper half of the surface.
But we 
establish
that not every doming
of an equiangular polygon derives from a slice, see Section~\secref{NotSlice}.

\begin{figure}[htbp]
\centering
\includegraphics[width=1.0\textwidth]{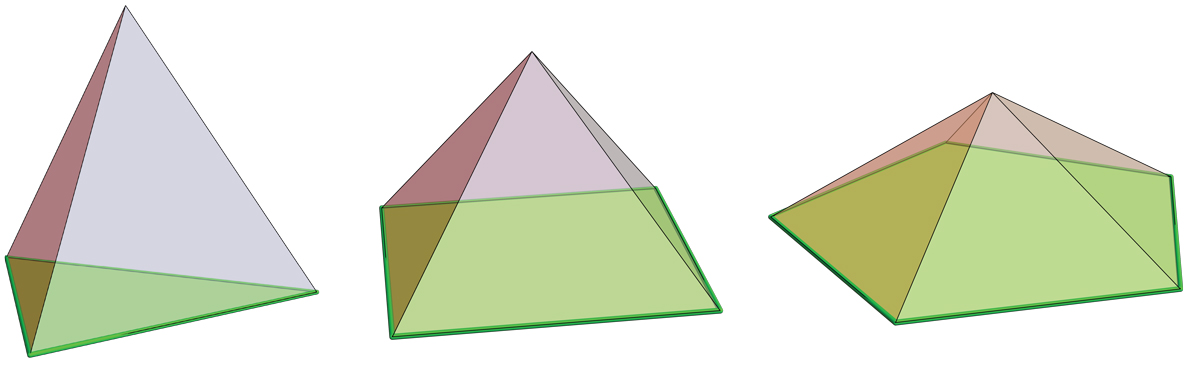}
\caption{Pyramids over $\bar{P}_n$, $n=3,4,5$.}
\figlab{Pyramids345}
\end{figure}

\begin{figure}[htbp]
\centering
\includegraphics[width=1.0\textwidth]{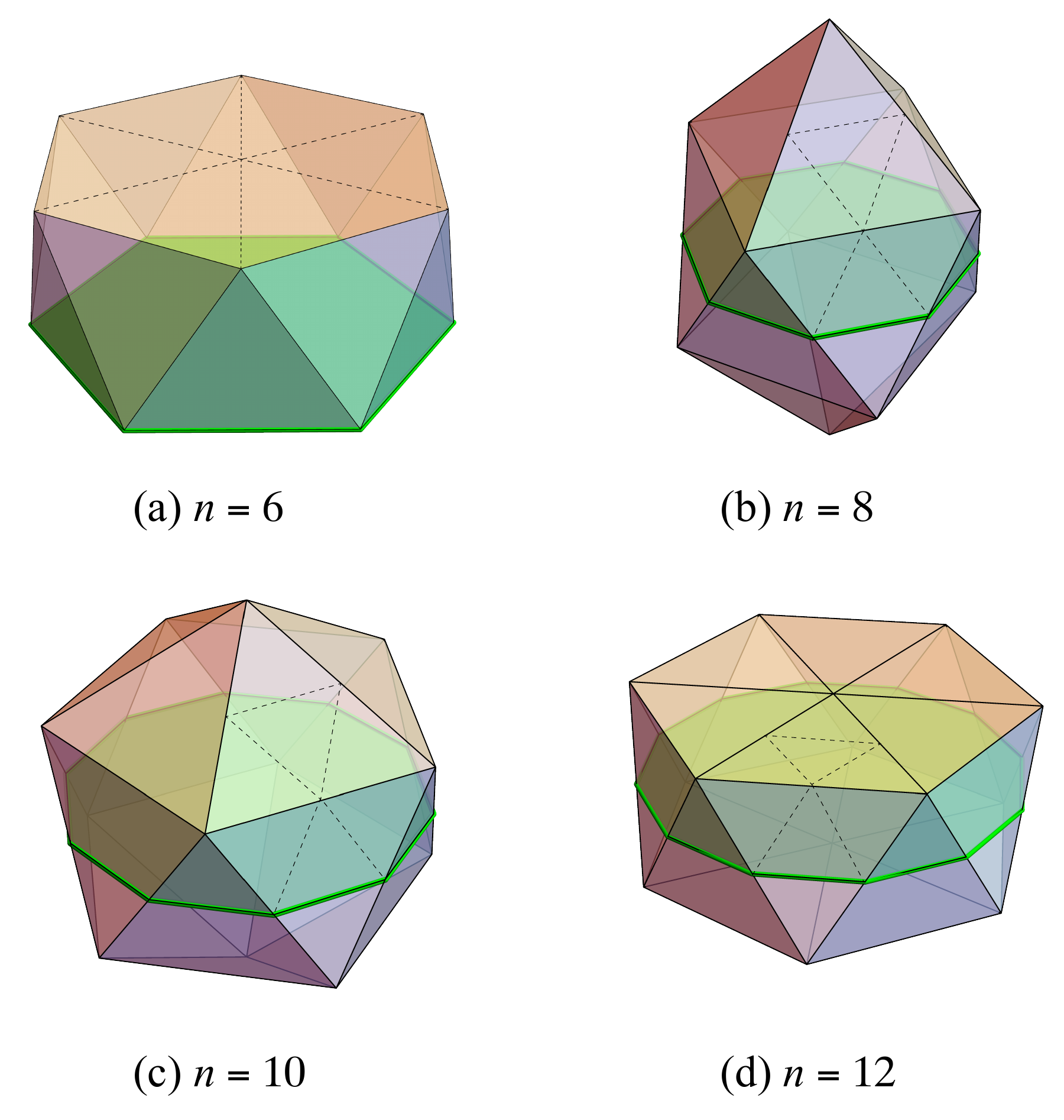}
\caption{Examples of $\bar{P}_n$ domes for $n=6,8,10,12$.
}
\figlab{SliceFigs}
\end{figure}

\begin{figure}[htbp]
\centering
\includegraphics[width=0.85\textwidth]{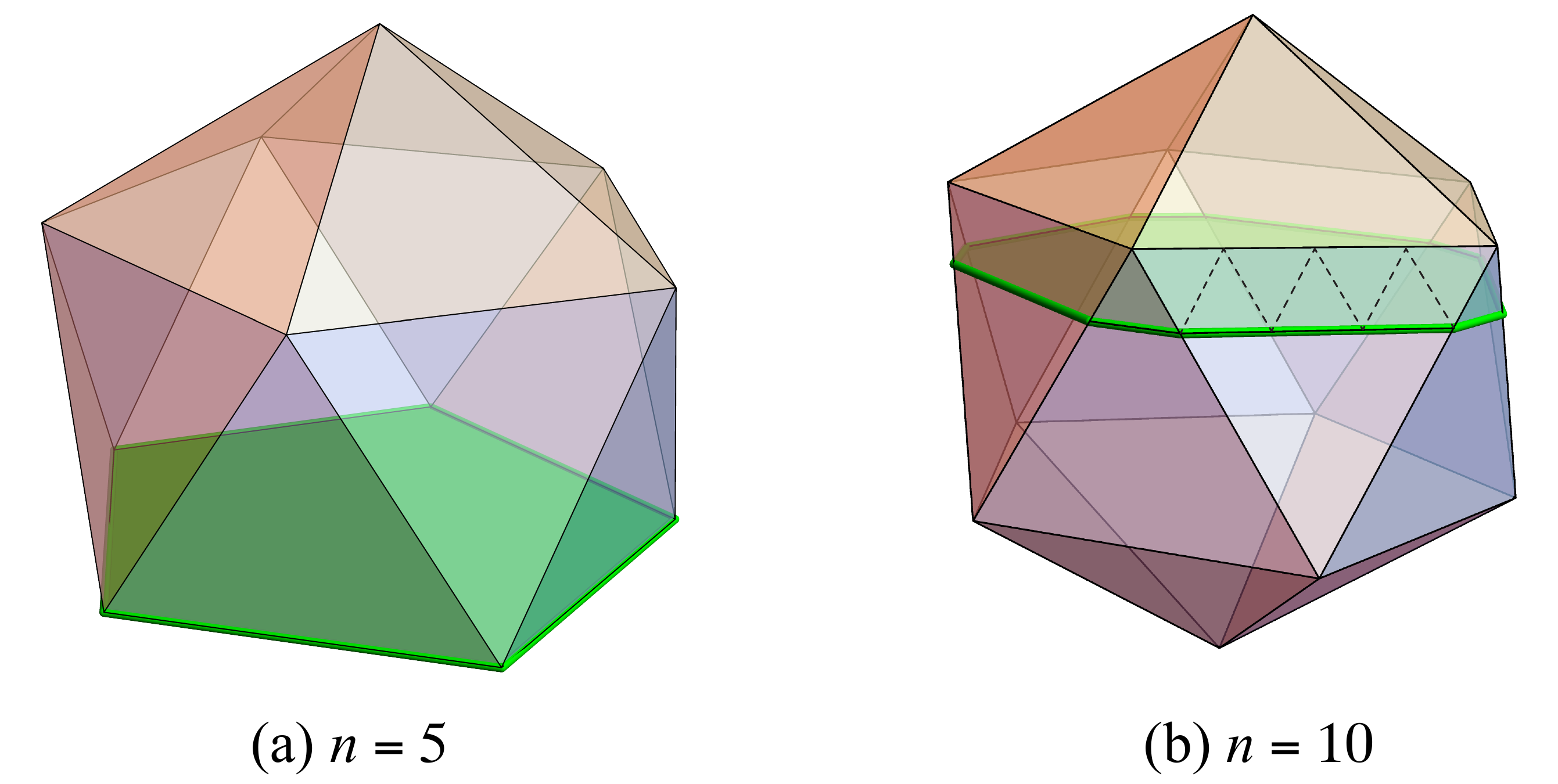}
\caption{(a)~A different dome over $\bar{P}_5$.
(b)~Doming an equiangular decagon with edge lengths alternating $1,3$.
}
\figlab{IcosaSlicesDifferent}
\end{figure}

Figures~\figref{Pyramids345} and~\figref{SliceFigs}
show one way to dome each regular polygon $\bar{P}_n$, but there are other solutions.
For example, $\bar{P}_5$ can be domed by a low slice through the icosahedron
as shown in
Fig.~\figref{IcosaSlicesDifferent}(a).
And again, these figures illustrate regular polygons, special cases of equiangular polygons.
To give a hint of the further possibilities, Fig.~\figref{IcosaSlicesDifferent}(b) shows 
how to dome an equiangular decagon
$P_{10}$ whose edge lengths alternate $1$ and $3$.



\section{Proof of Theorem~\thmref{EquiAng}(a): Restrictions on $n$}
\label{sec:part-a}
In this section we complete the proof of
the first half of Theorem~\thmref{EquiAng}:
The only equiangular 
$n$-gons 
that can be domed have $n \in \{3,4,5,6,8,10,12\}$.
For a dome over an equiangular $n$-gon, $n \ge 6$, we use the following steps: 
\begin{enumerate}[(1)]
\item Each base vertex has three incident dome triangles.

\item Curvature constraints imply that 
the number of (non-base) dome vertices is at most 6.

\item Of the $n$ dome faces incident to base edges, at least half tilt toward the outside of the base
and no two of these faces share a dome vertex.
Furthermore, for $n$ odd we strengthen this claim to \emph{all} dome faces incident to base edges. 

\item Thus, since there are at most 6 dome vertices, 
$n \le 12$, and for $n$ odd, 
there are no solutions for $n \ge 6$.
\end{enumerate}

Note that the base angle $\b$ of   
an equiangular $n$-gon
is $\frac{n-2}{n} 180^\circ$
so if $n \ge 6$, then every base angle is $\ge 120^\circ$. 
This weaker assumption 
on a domeable convex 
$n$-gon
is enough for most of our argument.

\begin{lemma}
\lemlab{3Triangles}
\label{lem:3Triangles}
If a 
base vertex  $b_i$ has base angle 
$\b_i \ge 120^\circ$,
then it is incident to three dome triangles. 
\end{lemma}
\begin{proof}
Base vertex $b_i$ cannot 
be incident to just one
or two triangles, otherwise the total face angle is $\le 120^\circ$, which
is not enough to span $\b_i$.
Vertex $b_i$ cannot
be incident to four triangles,
because 
$\b_i + 240^\circ \ge 360^\circ$,
and similarly for 
five (or more) triangles.
%
\end{proof}

From this we can analyze the base curvature:

\begin{lemma} 
\lemlab{BaseCurv}
If every base vertex $b_i$ is incident to three dome triangles,
then the sum of the curvatures at the base vertices is $2\pi$.
\end{lemma}
\begin{proof}
Let $\b_i$ be the angle of $P$ at vertex $b_i$.
Then the curvature at $b_i$ is $\o_i = 2 \pi - (\b_i + \pi)$, where the final $\pi$ term follows from
the assumption that $b_i$ is incident to three triangles.
Recalling that $\sum_i \b_i = \pi(n-2)$ for any simple polygon,
we have:
\begin{align*}
\sum_i \o_i  \ = \ \sum_i (\pi - \b_i) 
\ = \  n \pi - \sum_i \b_i 
 \ = \  n \pi - \pi (n-2) 
 \ = \  2 \pi.
\end{align*}
\end{proof}



%


\begin{lemma}
\lemlab{6dome}
\label{lem:6dome}
If $\D$ is a dome over a convex polygon $P$ that has all angles $\ge 120^\circ$, then $\D$ has at most $6$ dome vertices.
\end{lemma}
\begin{proof}
Let 
$V_3, V_4, V_5$ be the number of (non-base) dome vertices with $3,4,5$ incident triangles,
respectively.
%
By the Gauss-Bonnet theorem, the total curvature of a convex polyhedron is $4\pi$. 
The curvature of a $V_k$ vertex is $2\pi - k \frac{\pi}{3}$.
By Lemmas~\ref{lem:3Triangles} and~\lemref{BaseCurv} 
(this is where we use the assumption that all base angles are $\ge 120^\circ$)
the curvature at the base vertices of any dome over $P$ is $2\pi$. 
Thus, in units of~$\pi$:
\begin{align*}
V_3 +  \frac{2}{3} V_4 + \frac{1}{3} V_5 &= 2 \; .
\end{align*}
Thus
the number of dome vertices of $\D$ is 
$V_3 + V_4 + V_5 \le 3 V_3 +  2 V_4 +  V_5  = 6$. 
\end{proof}

\subsection{Face Normals and Private Dome Vertices}
\seclab{NormalsPrivate}


In this section we complete the proof of Theorem~\thmref{EquiAng}(a) by proving:

\begin{lemma}
\label{lem:12-and-6}
If $P$ is a domeable convex $n$-gon with all angles $\ge 120^\circ$, then $n \le 12$. 
Furthermore, there is no domeable equiangular $n$-gon for odd $n$ $ \ge 6$.   
\end{lemma}

To prove this, we 
consider a dome over $P$ and
follow step (3) which states that, of the $n$ dome faces incident to base edges, at least half of them [and for odd $n$, all of them] tilt toward the outside of the base
and no two of these faces share a dome vertex.
We begin by making this more formal.
Orient the dome with the base in the horizontal $xy$-plane. A dome triangle/face has an \defn{upward normal} if its outer normal has a 
positive $z$-component, and a \defn{downward normal} if its normal has a negative $z$-component.
Let $d$ be the number of dome faces with downward normals that are incident to base edges. 
Each of these faces 
is incident to
at least one dome vertex, and we prove that no two share a dome vertex.
This implies that $d \le 6$ (since there are at most $6$ dome vertices).  We then prove that  
$d \ge n/2$, which implies that $n \le 2d \le 12$, and we prove that for odd $n$, 
$d \ge n$, which implies that $n \le d \le 6$.


We first analyze upward and downward normals of the faces at a base vertex.

\begin{lemma}[$\pm$Normals]
\lemlab{Normals}
\label{lem:Normals}
Consider a base vertex $b_i$ with base angle $\ge 120^\circ$.  Suppose the three dome triangles incident to $b_i$
are $t_1, t_2, t_3$ where $t_1$ and $t_3$
are incident to the base edges at $b_i$
(possibly $t_2$ is coplanar with $t_1$ or with $t_3$, but not both).
See Fig.~\figref{3Tri_Overhead2D}.
Then 
$t_2$ has an upward normal 
and at least one of $t_1, t_3$
has a downward normal.
\end{lemma}
\begin{figure}[htbp]
\centering
\includegraphics[width=0.6\textwidth]{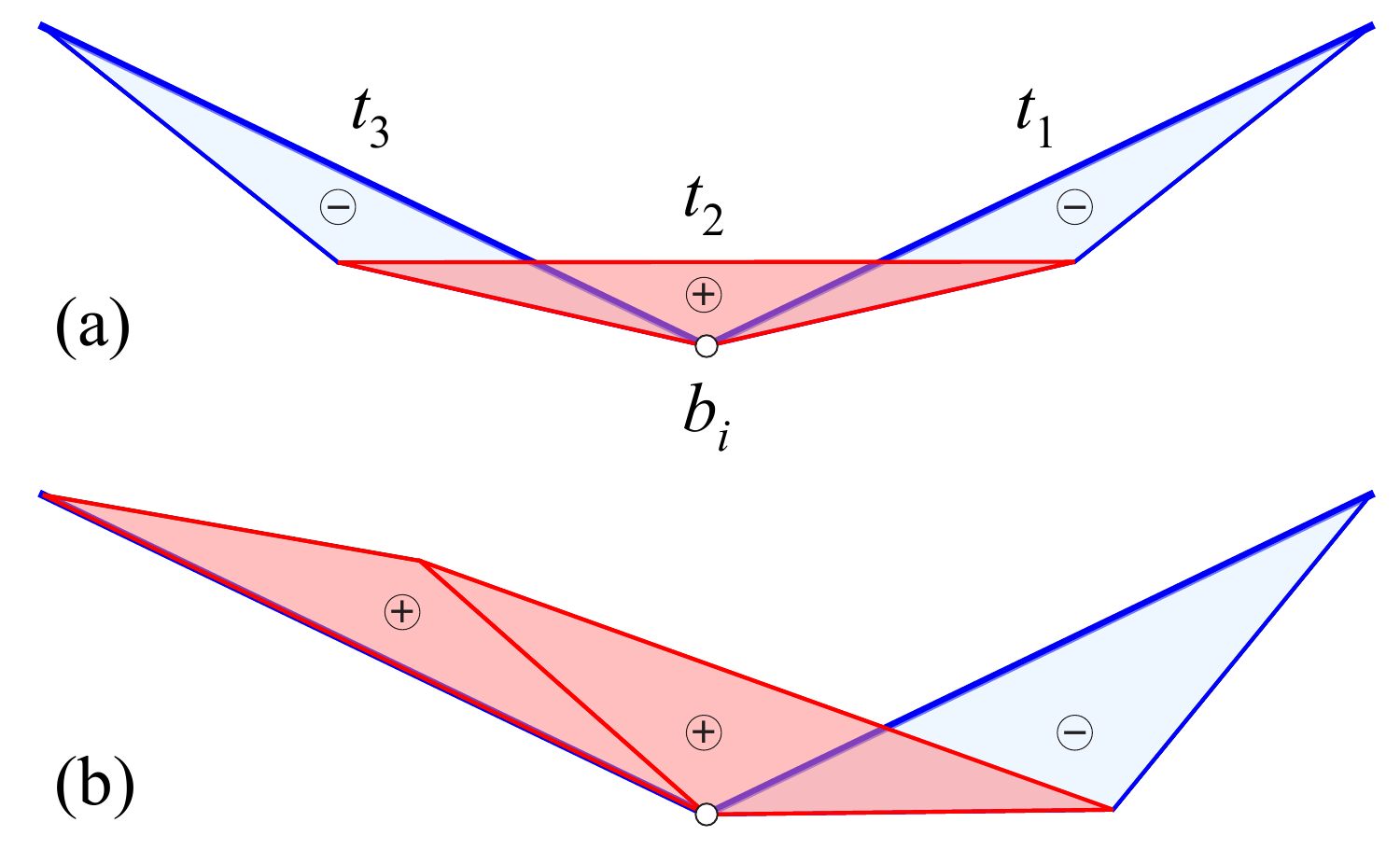}
\caption{
Overhead view of three  triangles
incident to base vertex $b_i$.
Triangles with upward normals pink, downward normals blue. (a)~Both $t_1$ and $t_3$ downward.
(b)~Only 
$t_1$ downward.
}
\figlab{3Tri_Overhead2D}
\end{figure}


\begin{figure}[htbp]
\centering
\includegraphics[width=0.8\textwidth]{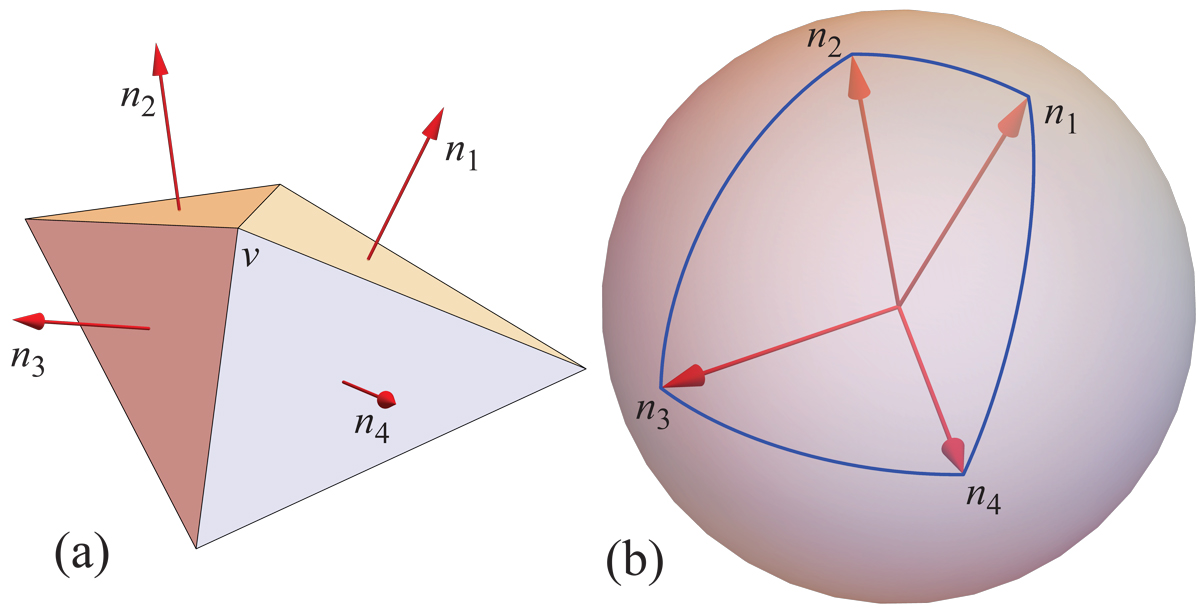}
\caption{The Gauss map for the vertex $v$.}
\figlab{GSph_generic}
\end{figure}

The argument depends on the Gauss map, which we first review.
We start with the Gauss map
for the neighborhood of 
a generic vertex $v$ of
degree $4$,
as described for example in~\cite{biedl2007cauchy}.
The four face normal vectors $n_i$, $i=1,2,3,4$, map to points
on the unit sphere $\S$.
This mapping can be viewed as placing the unit normals $n_i$ at the origin $o$
of $\S$, so the tips of $n_i$ are points on $\S$.
See Fig.~\figref{GSph_generic}.
Vertex $v$ and the edges incident to $v$ each map to their set of normals.  Thus the edge $e_i$ between the faces with normals $n_i$ and $n_{i+1}$ maps to an arc on $\S$ from $n_i$ to $n_{i+1}$, and $v$ itself maps to a spherical quadrilateral on $\S$. If $e_i$ has dihedral angle $\d_i$, then the corresponding arc has length $\pi - \d_i$.
For example, a $\d_i$ near $\pi$ indicates two nearly coplanar faces and
so two nearly parallel normals, leading to a near zero arc length.
Each face angle $\a_i$ incident to $v$ is an angle between two edges,
which on $\S$ becomes an angle between the corresponding 
arcs.
That angle is $\pi-\a_i$;
see, e.g.,~\cite{s-gpcpl-68}.
\medskip

%

\begin{proof} [Proof of Lemma~\lemref{Normals}]
Let $n_i$ be the normal of $t_i$, $i=1,2,3$ and let $n_0$ be the normal of the base.
The proof analyzes the Gauss map 
for vertex $b_i$;
see Fig.~\figref{GSph_Pent}.
Because each triangle $t_1, t_2, t_3$ has angle $60^\circ$ incident to $b_2$,
on $\S$ the angles at $n_1,n_2,n_3$ are each $\pi-60^\circ=120^\circ$.
The normal 
$n_0$ is at the south pole of $\S$.
Because 
the base angle at $b_i$ is $\ge 120^\circ$,
we know the angle at $n_0$ on $\S$ is $\le 60^\circ$.
We first prove that $n_2$ is an upward normal, i.e., it lies above the equator
of $\S$.

Construct on $\S$ the arc $x$ from $n_0$ to $n_2$ that lies inside 
the spherical
quadrilateral $n_0 n_1 n_2 n_3$. 
Note that  $|x| \le 180^\circ$  (where $|x|$ is the length of $x$) because the spherical quadrilateral is contained in a lune from south to north pole with angle $\le  60^\circ$ (the angle at $n_0$).
Thus, proving that $n_2$ lies above the equator is equivalent to proving that $|x| > 90^\circ$.





%
Arc $x$ splits the $120^\circ$ angle at $n_2$ into two parts, $A,B$, with $A$ on the $n_3$ side. 
Without loss of generality, assume $A \ge B$; so $A \ge 60^\circ$.  We also split the angle at $n_0$ into two parts $C,D$, with $C$ on the $n_3$ side.  We have $0 < C,D \le 60^\circ$. 

Apply the basic cosine law for spherical triangles
to the triangle $n_3 n_2 n_0$:

$$\cos 120^\circ = - \cos A \cos C + \sin A \sin C \cos |x| \;.$$

Rearranging and using $\cos 120^\circ = -\frac{1}{2}$ leads to

$$\cos |x| = \frac{- \frac{1}{2} + \cos A \cos C}{\sin A \sin C} \;.$$

Now $A \ge 60^\circ$ so $\cos A \le \frac{1}{2}$.  Also, $\cos C < 1$.  Thus the numerator is negative.  Because the 
sine terms are positive, the denominator is positive.  Thus $\cos |x| < 0$, 
and so $|x| > 90^\circ$. This establishes the first claim of the lemma: $n_2$ is an upward normal.

Now we turn to showing at least one of $n_1,n_3$ lies below the equator.
Let $y$ be the arc $n_0 n_1$.
Again apply the cosine law:

$$\cos B = - \cos 120^\circ \cos D + \sin 120^\circ \sin D \cos |y| \;.$$

$$\cos |y| = \frac{\cos B - \frac{1}{2} \cos D}{\sin 120^\circ \sin D} \;.$$

Now $B \le 60^\circ$ so $\cos B \ge \frac{1}{2}$.  Also, $\cos D < 1$.  Thus the numerator is positive.  Because the 
\anna{sine} 
terms are positive, the denominator is positive.  Thus $\cos |y| > 0$, so $|y| < 90^\circ$, i.e., $n_1$ is a downward normal. 
\end{proof}

\medskip
\noindent
The asymmetry derives from the assumption that $A \ge B$. So we draw no conclusion about $n_3$'s location on $\S$. And indeed it is possible for $n_3$ to be upward or downward.

\begin{figure}[htbp]
\centering
\includegraphics[width=0.50\textwidth]{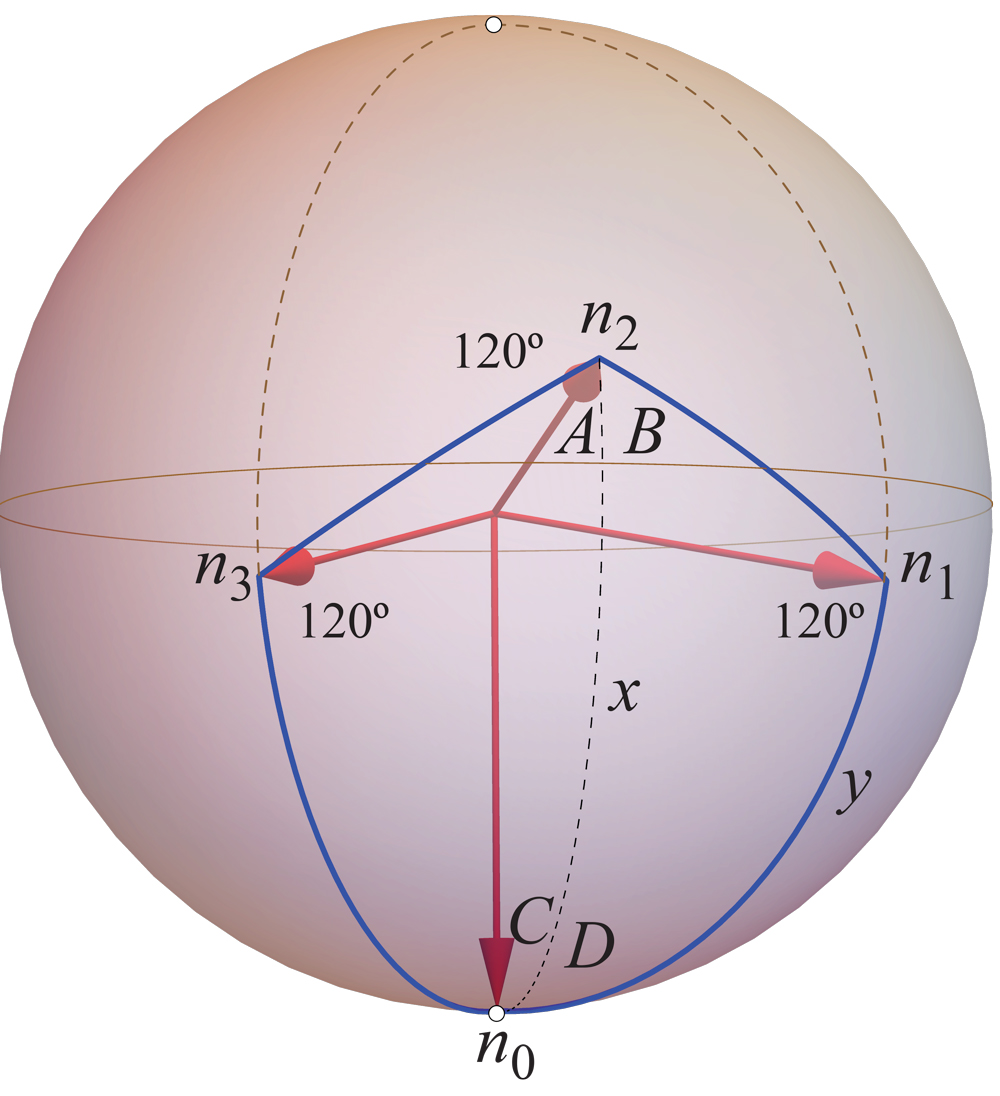}
\caption{The Gauss map for base vertex $b_2$. Here $n_2$ is an upward normal and
both $n_1,n_3$ are downward.
}
\figlab{GSph_Pent}
\end{figure}

Since every vertex of $P$ must be incident to at least one dome face with a downward normal, we immediately obtain the following result.
Recall that 
$d$ is the number of dome faces with downward normals that are incident to base edges.

\begin{corollary}
If $P$ is a domeable convex $n$-gon with all angles $\ge 120^\circ$, then $d \ge n/2$.    
\end{corollary}

Going forward, we need the following implication of the above proof.

\begin{obs} 
\label{obs:down-face}
In Lemma~\ref{lem:Normals}, if 
$t_1$ has a downward normal, then it cannot be coplanar with 
$t_2$
(which has an upward normal), and thus the 
dome face containing 
$t_1$ has a face angle of $60^\circ$ at the base vertex $b_i$.
\end{obs}

\begin{lemma}
\label{lem:d-le-6}
If $P$ is a convex $n$-gon with all angles $\ge 120^\circ$, and $\cal D$ is a dome over $P$, then $d \le 6$.
\end{lemma}
\begin{proof}
Consider a downward pointing face $f$ incident to base edge $e$. By Observation~\ref{obs:down-face}, $f$
has $60^\circ$ face angles at the endpoints of $e$, 
so the dome vertices of $f$, when projected to the $xy$-plane, lie in the equilateral $xy$-triangle on edge $e$.
See Fig.~\figref{3Tri_Overhead2D}
and Fig.~\figref{ProjectionPrivate}.
These triangles are disjoint, so 
no two downward pointing faces share a dome vertex.
Since $\D$ has
at most $6$ dome vertices (by Lemma~\ref{lem:6dome}) 
we have $d \le 6$.    
\end{proof}

\begin{figure}[htbp]
\centering
\includegraphics[width=0.5\textwidth]{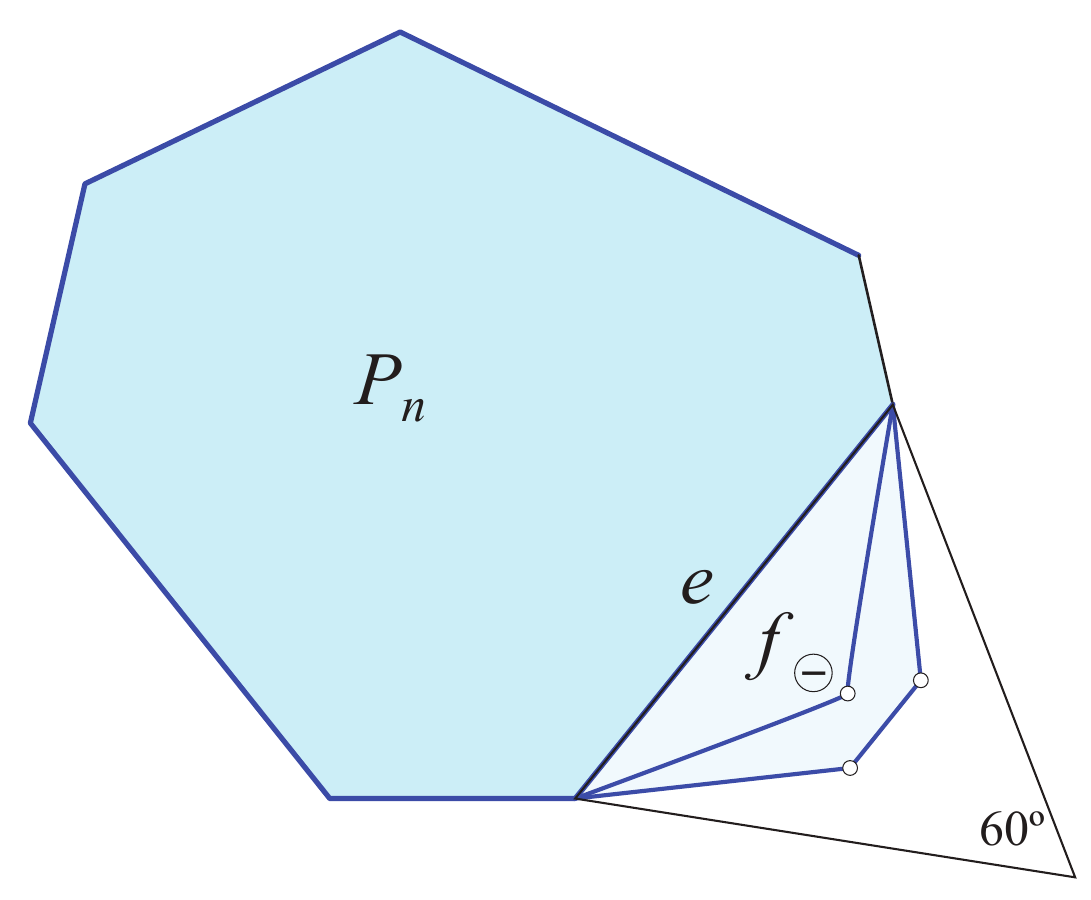}
\caption{Projection of downward face $f$ lies within an equilateral triangle outside base edge $e$. Here $P_n=P_7$.
}
\figlab{ProjectionPrivate}
\end{figure}

Finally, we will rule out odd $n > 6$ by proving the following lemma.

\begin{lemma}
\label{lem:all-down}
For any
dome over an equiangular $n$-gon with odd $n \ge 6$, %
$d \ge n$.
\end{lemma}


We begin with a preliminary result that will also be useful later on.
We henceforth often abbreviate ``dihedral angles'' by \defn{dihedrals}.
\begin{lemma}
\label{lem:dihedrals}
Let $P$ be an equiangular $n$-gon, $n \ge 6$ ($n$ may be odd or even) with edges $e_1, e_2, \ldots, e_n$.
Let $\D$ be a dome over $P$ and let $\d_i$ be the dihedral angle of $\D$ at base edge $e_i$.
If $n$ is even then the dihedrals $\d_i$ for $i$ even are all equal and the dihedrals $\d_i$ for $i$ odd are all equal.  If $n$ is odd then all dihedrals $\d_i$ are equal.
\end{lemma}
\begin{proof}
Let the vertices of $P$ be $b_i$, with $e_i = b_i b_{i+1}$. 
Since $P$ is equiangular and $n \ge 6$, the angle $\b$ at $b_i$ is $\ge 120^\circ$. 
By Lemma~\lemref{3Triangles}, there are three dome triangles $t_{i,1}, t_{i,2}, t_{i,3}$ incident to $b_i$, where $t_{i,1}$ is incident to $e_{i-1}$ and $t_{i,3}$ is incident to $e_i$;
see Fig.~\figref{Claim37}.  Each triangle has a $60^\circ$ angle at $b_i$. Together with the base angle of $\beta$, this gives---in isolation---one degree of freedom for the dihedral angles 
of edges incident to $b_i$, ranging from coplanarity of the first two triangles to coplanarity of the last two triangles.

Because there is one degree of freedom at $b_i$, 
fixing $\d_i$ then fixes the
dihedral angles of all edges 
incident to $b_i$ and to $b_{i+1}$. Furthermore, by reflective symmetry of $b_i$ and $b_{i+1}$ (which have the same face angles), 
we have $\d_{i-1}=\d_{i+1}$.
This implies the result.
\end{proof}

\begin{proof} [Proof of Lemma~\ref{lem:all-down}]
By Lemma~\ref{lem:dihedrals}, all the dihedrals of $\D$ at base edges of $P$ are equal. 
By Lemma~\lemref{Normals}, at least one dome face incident to a base edge has a downward normal (equivalently, a dihedral angle $> 90^\circ$). Thus all the dome faces incident to base edges must
have downward normals.
\end{proof}

This completes the proof of Lemma~\ref{lem:12-and-6}.

\begin{figure}[htbp]
\centering
\includegraphics[width=0.60\textwidth]{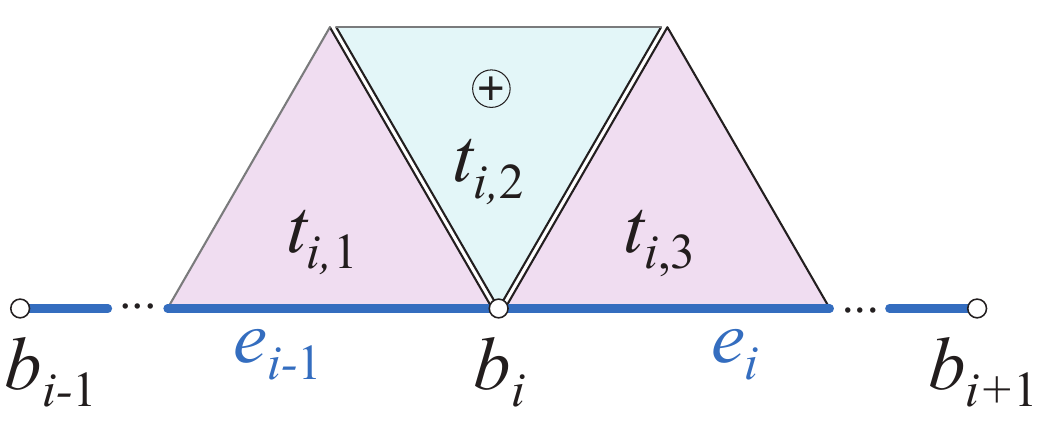}
\caption{Dome triangles incident to base vertex $b_i$.
}
\figlab{Claim37}
\end{figure}

\section{Proof of Theorem~\thmref{EquiAng}(b): Restrictions on edge lengths}
\label{sec:main-part-b}

We now turn to constraints on the edge lengths of equiangular polygons that can be domed,
part~(b) of our main theorem.
Of course all edge lengths must be 
integers
because the dome is composed of unit equilateral triangles.
For $n \le 6$, the constraints are relatively easy to identify. 
For $n=8,10,12$, significant argument is needed to establish if-and-only-if characterizations.
We first examine the edge length conditions in more detail.

\subsection{List of Conditions}
\seclab{ListConditions}
Recall the conditions stated in the main theorem:
for $n=3,4,5,6$ every equiangular integer polygon can be domed, and for $n=8,10,12$ an equiangular integer $n$-gon can be domed iff the lengths of the odd-numbered edges are equal and the even-numbered edge lengths are those of an equiangular $\frac{n}{2}$-gon.

Combining this with the known  characterizations of equiangular integer $n$-gons for $n=3,4,5,6$~\cite{ball2002equiangular}, 
we obtain the following:
\begin{description}
\item[$n=3$:] All side lengths equal.
\item[$n=4$:] Side lengths $a,b,a,b$, i.e., all 
integer rectangles. See Lemma~\lemref{Rect}.
\item[$n=5$:] All side lengths equal. This follows because every  
integer equiangular pentagon is regular~\cite{ball2002equiangular}.
\item[$n=6$:] Side lengths are
those of a polyiamond $6$-gon, namely, $a,b,c,a',b',c'$ with $a-a' = b'-b = c-c'$:
Fig.~\figref{Polyiamond6}.
\item[$n=8$:] Odd edges have equal length $a$.
Even edges have lengths $b,c,b,c$. 
\item[$n=10$:] Odd edges have equal length $a$, even edges have equal length $b$. 
For example, see the earlier Fig.~\figref{IcosaSlicesDifferent}(b).
\item[$n=12$:] Odd edges have equal length $d$. 
%
Even edges are those of an equiangular $6$-gon, i.e., lengths are $a,b,c,a',b',c'$ with $a-a' = b'-b = c-c'$.
\end{description}

We aim to prove that an equiangular $n$-gon can be domed if and only if
the above constraints are satisfied.
The cases $n=3,4,5$ are already settled.

\subsubsection{$n=6$}
\seclab{n=6}

We prove the following more general result:

\hide{For $n=6$, the side-length constraints imply that $P_6$ is a polyiamond,
a truncation of the three corners of an equilateral triangle: Fig.~\figref{Polyiamond6}.
\begin{align*}
c' + a + b  = \; & b + c + a' \\
a -a'  = \; &  c - c' \\
a' + b' + c'  =\;  & b + c + a' \\
b' - b  =  \;& c - c'
\end{align*}
} 

\begin{lemma}
\lemlab{Polyiamond}
Every convex polyiamond can be domed.
\end{lemma}
\begin{proof}
A polyiamond on $3,4,5,6$ sides is formed by cutting off $0,1,2,3$, respectively, corners (small equilateral triangles) off a large integer equilateral triangle $T$; see Fig.~\figref{Polyiamond6}.
Construct a regular tetrahedron with triangle $T$ as one face.
Fig.~\figref{fig-polyiamonds} shows the case of a 6-sided polyiamond.
Extend the up-to-three small cut-off triangles to regular tetrahedra, shaded gray in the figure.  Cutting the gray tetrahedra off the large tetrahedron results in a truncated tetrahedon 
with the initial (pink) polyiamond as one face.
All the other faces are integer polyiamonds as well.
\end{proof}

\begin{figure}[htbp]
\centering
\includegraphics[width=0.4\textwidth]{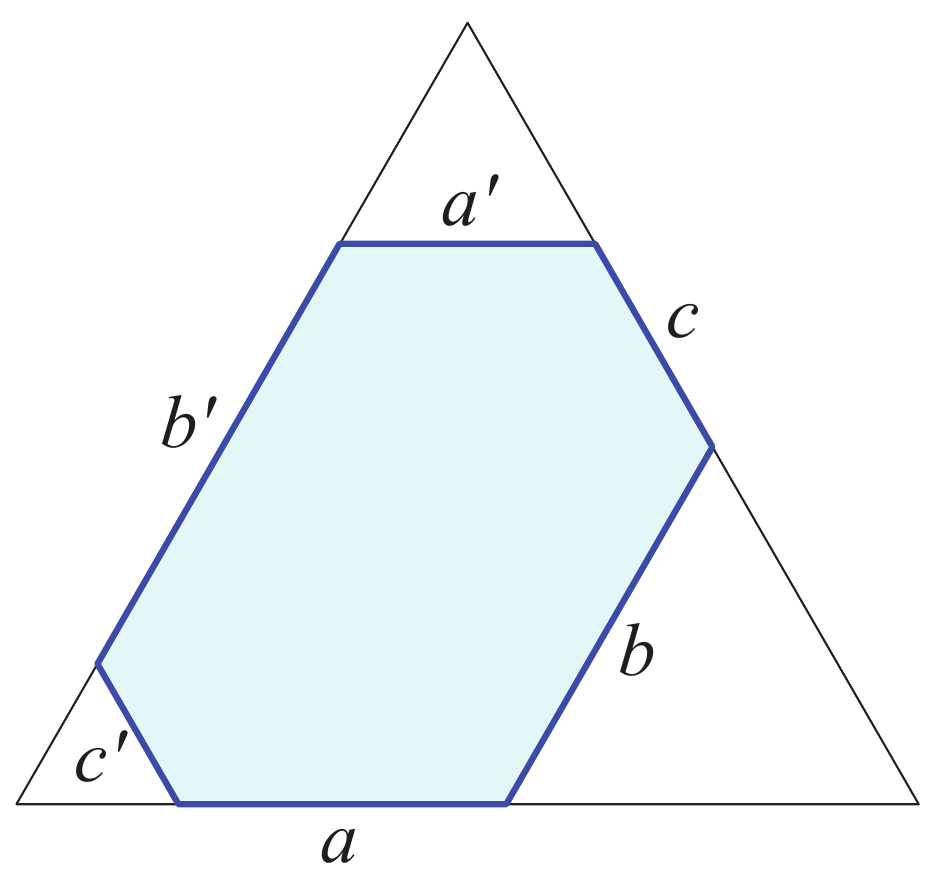}
\caption{A polyiamond in an integer equilateral triangle
of side length $a+b+c'$ (here any of $a', b, c'$ may be $0$). 
}
\figlab{Polyiamond6}
\end{figure}

\begin{figure}[htbp]
\centering
\includegraphics[width=0.4\textwidth]{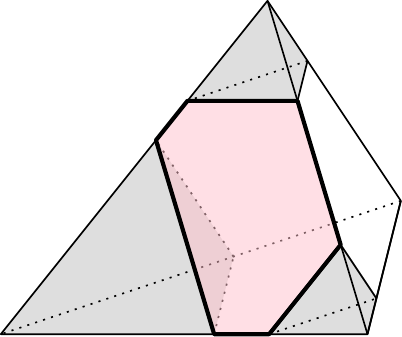}
\caption{The (pink) polyiamond is domed by the truncation of three corner tetrahedra.}
\figlab{fig-polyiamonds}
\end{figure}


\subsection{Conditions $\to$ domeable} 
For $n=8,10,12$, we first show that the edge-length conditions listed above
lead to domings of each $P_n$---the ``if'' direction of the claim.
The construction begins the same way in all three cases.
Suppose the odd base edges have length $\ell$.  Erect an equilateral triangle $t$ with side length $\ell$ on each odd base edge.
%
This triangle has a downward normal.
Two consecutive triangles are joined by
an upward facing trapezoid.
Then at the height of the apex of $t$,
we have a planar polygon on $n/2$ vertices---a ``lid.''
For an octagon, 
the 
lid is a rectangle, which we then dome.
For a decagon,
the 
lid is a regular pentagon, which we then dome.
For a dodecagon,
the 
lid is a polyiamond hexagon, which completes the dome. 
Below we give further details and show that the resulting polyhedron is convex in each case.

\subsubsection{$8$-gons $\to$ domeable}
We dome any octagon satisfying the edge-length constraints as illustrated
in Fig.~\figref{Octa_4x2}. The four downward pointing faces are tilted at
dihedral angle the same
as the corresponding faces of a square antiprism 
(the case that arises from a regular octagon).
One level of triangles results in a rectangular \anna{lid},
which we then ``roof'' as 
seen in Figs.~\figref{RectInt} and~\figref{Rect_3x1_roof}.
Combining the dihedral angles of a square antiprism
($\approx 104^\circ$) and a square pyramid
($\approx 55^\circ$)
leads to a dihedral angle where the roof meets a
side of $\approx 159^\circ < 180^\circ$ and so is convex.

\begin{figure}[htbp]
\centering
\includegraphics[width=0.6\textwidth]{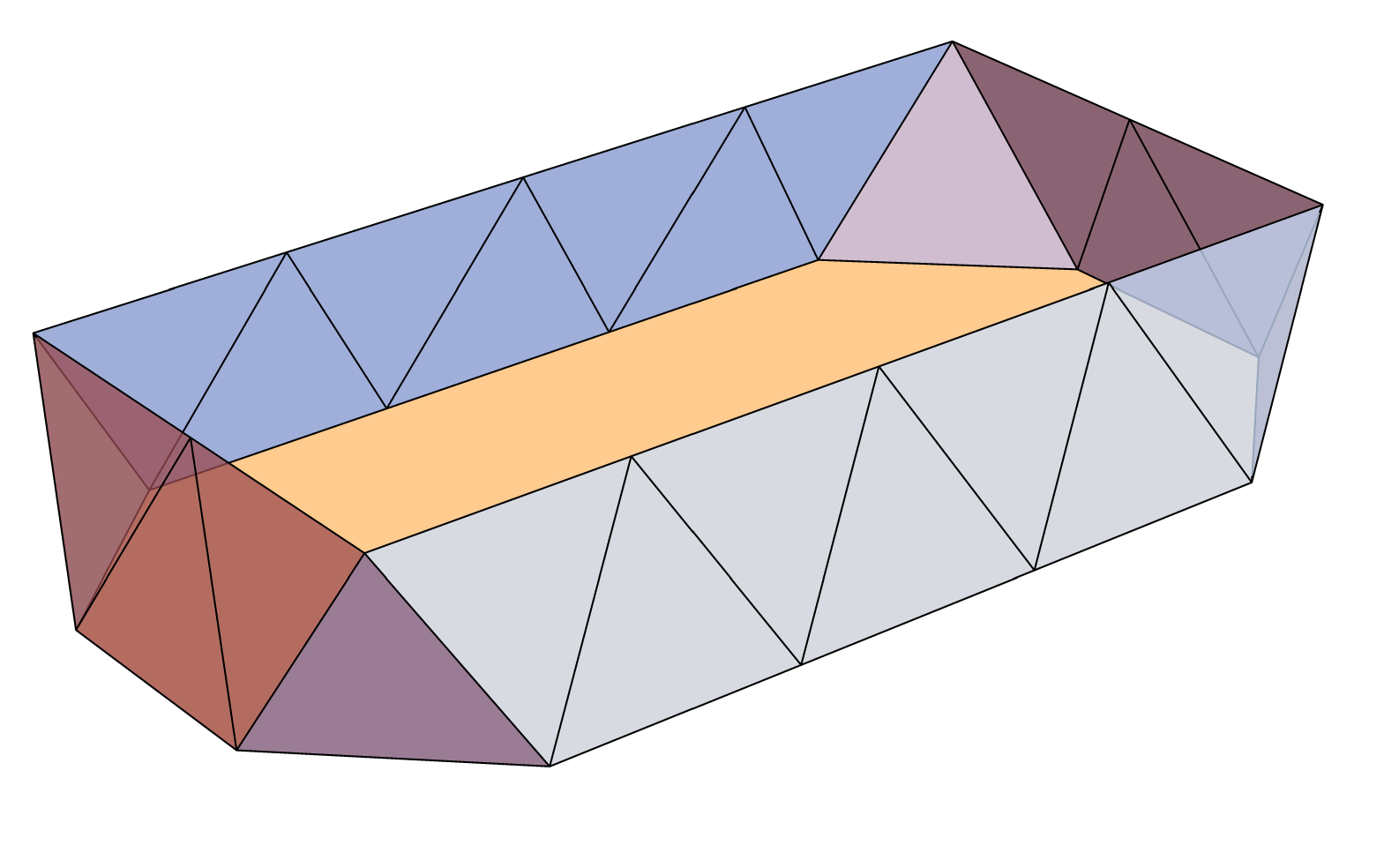}
\caption{An octagon base with edge lengths $1,4,1,2,1,4,1,2$.
The 
``lid'' is a $4 \times 2$ rectangle.}
\figlab{Octa_4x2}
\end{figure}

\subsubsection{$10$-gons $\to$ domeable}
For certain values of $a,b$,
the length pattern $a,b,\ldots,a,b$ can be achieved by adjusting the height of the
slice through an icosahedron; see
Figs.~\figref{SliceFigs}(c) and~\figref{IcosaSlicesDifferent}(b).
But the more general strategy was described above:
At the height of the $a$-triangle's apex,
the lid is
a regular pentagon with side length $b$.
This can then be domed 
with a pentagonal pyramid
as in Fig.~\figref{Pyramids345}.
Combining the dihedral angles of a pentagonal antiprism
($\approx 101^\circ$) and a pentagonal pyramid
($\approx 37^\circ$)
leads to a dihedral angle where the top 
meets a
side of $\approx 138^\circ < 180^\circ$ and so is convex.

\subsubsection{$12$-gons $\to$ domeable}
The odd edges of equal length $d$ are covered by downward facing equilateral $d$-triangles.
See Fig.~\figref{PepaConstruction}.
One level up, 
the lid is
a polyiamond $6$-gon, 
which can be filled in to create
the last face of the dome.
\begin{figure}[htbp]
\centering
\includegraphics[width=0.5\textwidth]{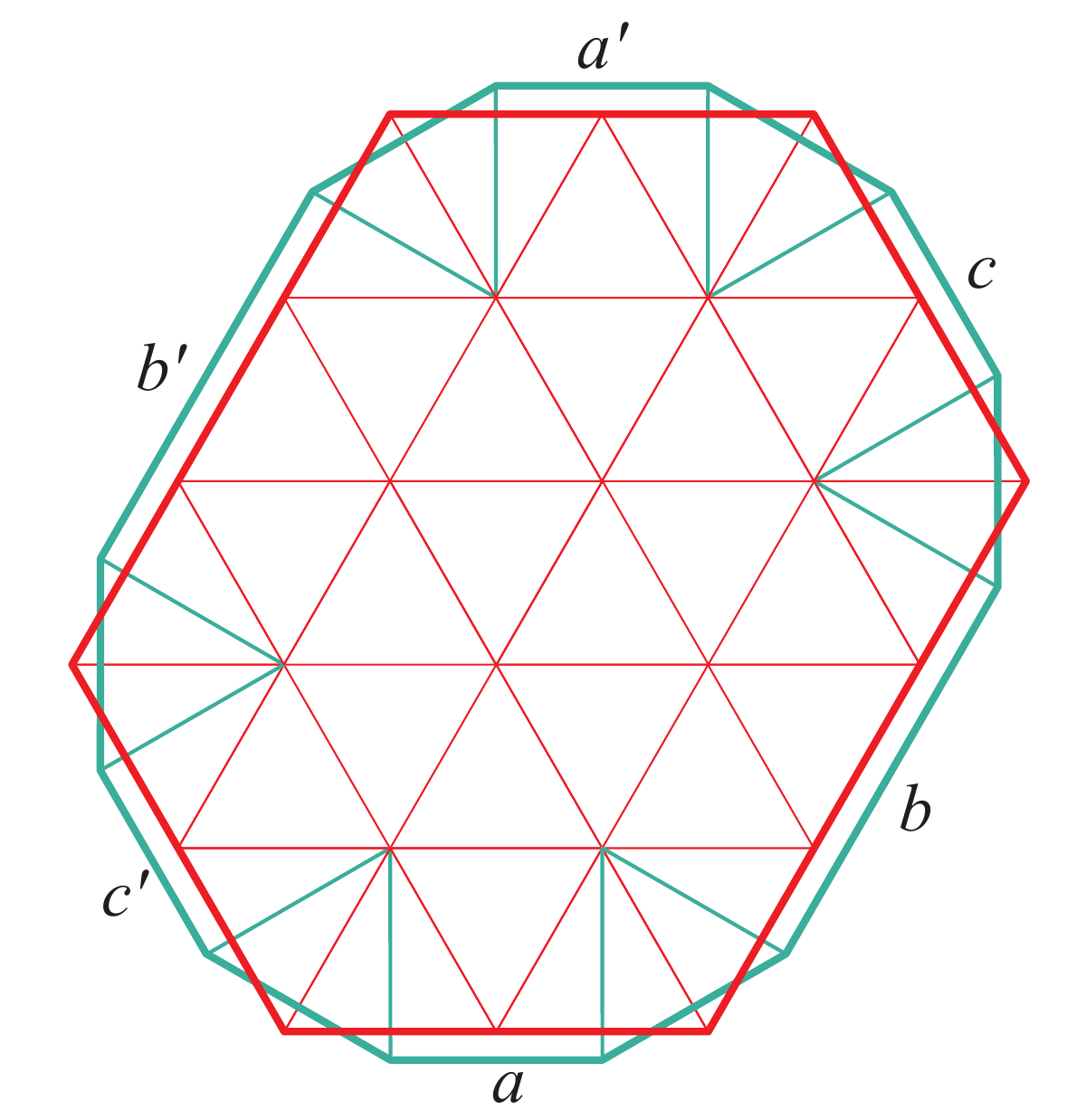}
\caption{Equiangular $12$-gon.
Red hexagon is polyiamond roof.
Green polygon is base $P_{12}$, angle $\b = 150^\circ$.
$a,b,c=1,2,1$; $d=1$.}
\figlab{PepaConstruction}
\end{figure}

\subsection{Domeable $\to$ conditions}

In this section we prove
that any domeable equiangular integer $n$-gon, $n=8,10,12$, satisfies the edge length conditions. We know from Lemma~\ref{lem:6dome} that
a dome $\D$ over such an $n$-gon has
at most 6 dome vertices. We will prove properties of 
the 
dome to derive the edge length conditions.


\subsubsection{$12$-gons $\to$ conditions}
We first handle the $n=12$ case, as it is simpler than $n=8,10$.
We will repeat some of the logic when addressing $n=8,10$ below.
We present the argument in five steps.  Refer throughout to Fig.~\figref{n12Proof}.

\begin{enumerate}[(1)]
\item
As we saw in Lemma~\ref{lem:dihedrals},
the dihedrals at even base edges are equal, and the dihedrals at odd base edges are equal.  By Lemma~\ref{lem:Normals}, one set, say the 
odd
edges, 
has dihedrals $> 90^\circ$, i.e., the corresponding dome faces have downward normals.
\begin{figure}[htbp]
\centering
\includegraphics[width=0.65\textwidth]{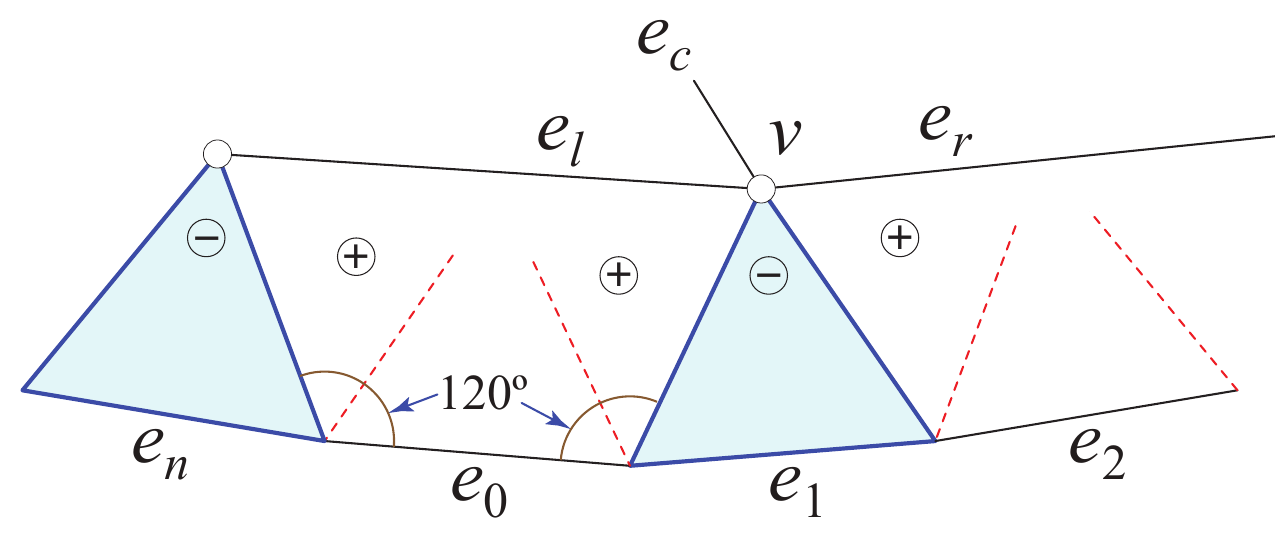}
\caption{Faces on odd
edges are triangles pointing down.
Faces on even 
edges are quadrilaterals.
}
\figlab{n12Proof}
\end{figure}
%
\item The projection argument in 
Lemma~\ref{lem:d-le-6}
shows that 
the $6$ faces incident to odd base edges have disjoint sets of dome vertices.
Since 
our dome $\D$ has
at most $6$ dome vertices, 
each of these faces gets one 
dome vertex, and so it must be a triangle.
Also, each dome vertex is the apex of one of these triangles.
%
\item Each base vertex \anna{$b$} has three incident dome triangles.  One of them is in the downward-normal face at the vertex.  The edge incident to $b$ that lies between the other two dome triangles will be called a ``red'' edge (dashed in Fig.~\figref{n12Proof}).
Next we prove that the 12 red edges 
are flat, i.e., have a $180^\circ$ dihedral.
By symmetry
of the base vertices, all the red edges have the same dihedral, so if one is nonflat, then they all are.  
What are the 
dome vertices at the other ends of these nonflat red edges? (Note that a red edge cannot go to another base vertex.)
Since there are $12$ red edges and only $6$ dome vertices,
some dome vertex $u$ is an endpoint of at least two red edges. 
As noted above, $u$ is the apex of a downward-normal triangle, so
it 
also has two (different) edges to the endpoints of an odd
base edge. Because $u$ has at most $5$ incident dome triangles, 
it has at most one more incident dome edge
beyond the two red edges and two edges to the base.
This means that two of $u$'s edges separated by $60^\circ$
(thus in one face)
are incident to non-consecutive base vertices. 
But it is impossible for two non-consecutive base vertices to lie in the same dome face.
Thus the red edges cannot end at dome vertices.
Thus they must be flat edges.
The face on any 
even 
base edge then has $120^\circ$ face angles at the base vertices.  We next argue that each such face is a quadrilateral.
\item Consider now a dome vertex $v$, the apex of a triangle face on an odd
edge, say $e_1$ in the figure.
Two of its (at most) five incident edges go to the base endpoints of 
$e_1$. 
And the two 
incident edges to the
left and right, $e_l$ and $e_r$,
(whether flat or not) must be horizontal and parallel to the corresponding base edge
$e_0$ and $e_2$.
The angle between $e_l$ and $e_r$ is then $120^\circ$.
Therefore $e_c$, the central edge between them, must be flat, and $e_l$ and $e_r$ are not flat.
\item 
The face on any even
base edge $e_i$ is then a trapezoid 
with $120^\circ$ angles at the base and $60^\circ$ angles at the dome vertices.   
Then the dome vertices must all be at the same level, 
so the triangles on odd base edges all have the same edge length, say $\ell$. 
The length of the top edge of the trapezoid on $e_i$ is $|e_i| + \ell$.
These top edges 
joining the dome vertices then form an equiangular integer hexagon, a polyiamond, whose edge lengths must then be $a,b,c,a',b',c'$ with $a-a' = b'-b=c-c'$.  
Thus the 
even  
base edge lengths also satisfy that constraint, which proves the result for $n=12$.
\end{enumerate}

\subsubsection{$8,10$-gons $\to$ conditions}
\seclab{n810Proof}
Compared to the case $n=12$ there are some hurdles to overcome. 
One is that a 
downward-normal face on an odd  edge 
might
not be a triangle---it 
could
be a quadrilateral with two
dome vertices.
Another is that there may be 
dome vertices that are not included in any of the faces on the odd edges,
at most one for $n=10$ and at most two for $n=8$.  
Despite these difficulties, we will still 
follow the same outline as above, first proving
that the red edges in Fig.~\figref{n12Proof} are flat, and then analyzing the situation at one dome vertex.

\hide{These two cases are more difficult than $n=12$, because we can no longer argue (in step (2) above)
\Anna{Should we say: that the faces on even edges are triangles?  I mean, aren't we stuck even if all dome vertices are private?}
\JOR{Not sure how to respond to this. The original seems
to me to be true, and reads well.
This is just an introduction, so we don't need
absolute precision. As long as the steps to follow
are correct.}
that every dome vertex is private,
which was crucial for proving that red edges are flat.
For $n=10$ we may have one non-private dome vertex.  For $n=8$ we could have two non-private dome vertices.
} 

\begin{enumerate}[(1)]
\item
As we saw in Lemma~\ref{lem:dihedrals},
the even 
base edges have equal dihedrals, and the odd 
base edges have equal dihedrals and are incident to faces with downward normals.
%
These downward faces on the odd 
edges, which we call 
\emph{blue}
faces, have $60^\circ$ angles at the base
by Observation~\ref{obs:down-face}. 
Thus each blue face is
either a triangle
with one 
dome
vertex, 
or a trapezoid with two 
dome
vertices (again, following the projection argument in 
Lemma~\ref{lem:d-le-6}).
See Fig.~\figref{RedFaces_FigB}.
For $n=12,10,8$, we need at least $n/2=6,5,4$
of the $6$ dome vertices to lie in blue faces.
For $n=12$, 
this means that the 
blue
faces are all triangles.
For $n=10,8$, there could be one or two
blue faces that are trapezoids.
\item 
As for the $n=12$ case above, we identify a ``red'' edge at each base vertex $b$---the edge between the two dome triangles at $b$ that are not in the blue face. As before, we will argue that all red edges are flat. 
Suppose not.  As above, by symmetry, this implies that all the red edges are nonflat.
Here is where matters are more complicated
than in the $n=12$ case.
Consider the faces incident to the even 
base edges. 
Because the red edges are nonflat, each such face is bounded by red edges at the base and we call it a \emph{red} face.  A red face has  
$60^\circ$ degree face angles at the base, so it must be a triangle with one dome vertex, or a trapezoid with two dome vertices.
See Fig.~\figref{RedFaces_FigB}.
Note that there is an upward-normal face between consecutive red and blue faces.

\begin{figure}[htbp]
\centering
\includegraphics[width=0.55\textwidth]{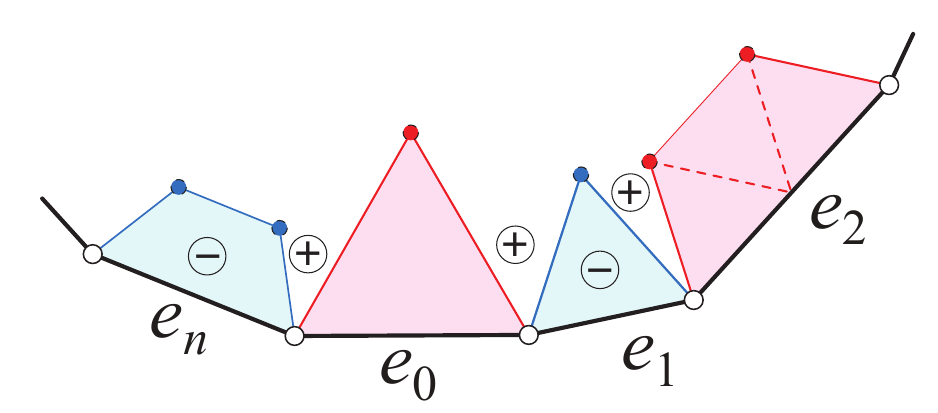}
\caption{
If red edges are nonflat, then faces incident to the base alternate between red faces and blue 
downward faces, each with one or two dome vertices.
}
\figlab{RedFaces_FigB}
\end{figure}

Our next goal is to show that 
a dome vertex cannot lie in more than one red or blue face.
This immediately gives a contradiction, since we need at least 
one dome vertex for each of the at least $8$ base edges
but there are at most $6$ dome vertices.  

\begin{claim}
A dome vertex cannot be shared by two red/blue faces.
\end{claim}

\begin{proof} Suppose dome vertex $v$
is shared by 
two faces, either of which may be red or blue,
say on base edges $e_i$ and $e_j$.
See Fig.~\figref{n810RedFace}.
In case either of these faces is a trapezoid, we focus on the equilateral triangle (a union of dome triangles) that is contained in the trapezoid and  has its top $60^\circ$ corner at $v$ and its other two ``base'' corners on the base edge.  
We consider what happens between the two triangles.  If they share a side, then the common base corner is a
base vertex with only two incident dome triangles, a contradiction.
Next, suppose 
there is a $60^\circ$ angle at $v$ between the two triangles. Then this angle is part of one dome face  that extends to a base corner of each of the two triangles. This is only possible if the two base points lie on a common base edge joining $e_i$ and $e_j$.  But then we have base vertices with only two incident dome triangles, 
again a contradiction.   
Thus there must be at least $120^\circ$ of surface angle at $v$ on each side between the two triangles, a contradiction to $v$ having at most $5$ incident dome triangles.
\end{proof}
%
\begin{figure}[htbp]
\centering
\includegraphics[width=0.75\textwidth]{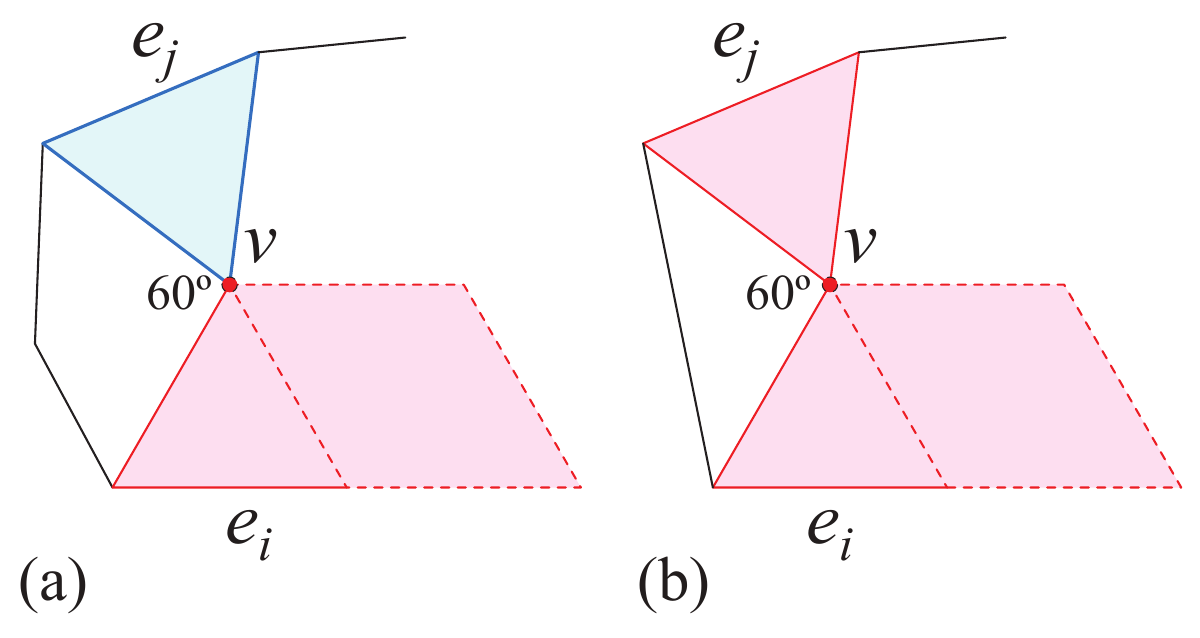}
\caption{(a)~A dome vertex $v$ cannot be in a red face and a blue face.
(b)~Two red faces cannot share a vertex $v$.
}
\figlab{n810RedFace}
\end{figure}


This completes the proof that the red edges are flat.
We conclude that
faces on even 
base edges have
$120^\circ$ angles at the base, 
and upward normals, as shown
in Fig.~\figref{n12Proof} (except that the blue faces need not be triangles).
\item 
The next goal is to prove that every blue face is a triangle and that the other faces on base edges are trapezoids.  
We accomplish this by analyzing
the local structure at a 
dome vertex $v$ in a blue face.
%

Vertex $v$ lies in a blue downward face $f$,
say on base edge 
$b_1 b_2 = e_1$.
Suppose, without loss of generality, that 
$v b_1$
is an edge of face $f$; then segment 
$v b_2$
lies in $f$ but is an edge of $f$  only if $f$ is a triangle.
Name the edges (flat or nonflat) between the dome triangles at $v$ as shown in Fig.~\figref{PrivateVertex}(a):
in the 
clockwise
direction from 
$v b_1$
is the edge $e_l$; in the counterclockwise
direction, there is an edge in $f$, then the edge $e_r$; and if there is a fifth edge, it is $e_c$ between $e_r$ and $e_l$.  
The upward face sharing edge 
$v b_1$
has face angle $120^\circ$ incident to 
$b_1$.
Edge $e_l$ is in or on the boundary of this face (we do not assume the face is a quadrilateral), and
thus $e_l$ is horizontal and parallel to 
$e_0$.
%
\begin{figure}[htbp]
\centering
\includegraphics[width=1.0\textwidth]{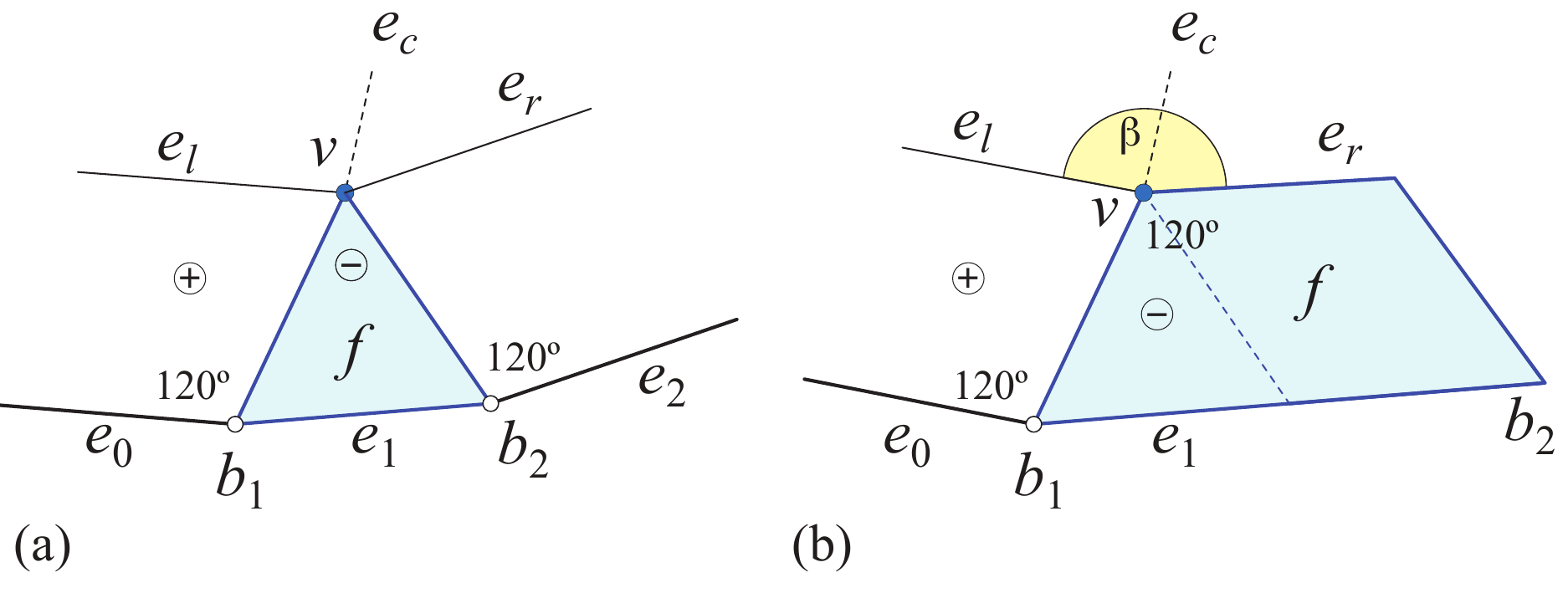}
\caption{Dome vertex $v$ in a blue face $f$. 
Face $f$ cannot be a quadrilateral.
}
\figlab{PrivateVertex}
\end{figure}
\item We next rule out 
face $f$ being a trapezoid.
Suppose it is. 
Then the angle of $f$ at $v$ is $120^\circ$ and edge $e_r$ is an edge of trapezoid $f$, thus horizontal and parallel to $e_1$.
See Fig.~\figref{PrivateVertex}(b).
But this means that the angle between $e_l$ and $e_r$ is the same as the base angle $\b$
which is greater than $120^\circ$.
Even if $v$ 
is 
a $V_5$ vertex,
there is only one edge $e_c$ between $e_l$ and $e_r$,
not enough to cover $\b$.
Therefore $f$ is a triangle
$b_1 b_2 v$, and
edge $e_r$ lies outside $f$.
\item 
The same argument as above for $e_l$ then shows that
$e_r$ is  horizontal and parallel to 
\anna{$e_2$}.
The angle between $e_l$ and $e_r$ is then $90^\circ$, $108^\circ$, or $120^\circ$ for $n=8,10,12$ respectively.  This angle is too large to be covered by a single triangle,
so it requires at least two triangles.
Therefore $v$ must be a $V_5$ vertex,
and the edge $e_c$ is uniquely determined 
and the edges $e_l$ and $e_r$ are nonflat.
\item 
To wrap up,
all the blue downward 
faces are triangles, and
their dome vertices
are $V_5$ vertices.
The faces on the even 
base edges are 
trapezoids with face angles of $120^\circ$ at the base and angles of $60^\circ$ at the dome vertices. 
(Recall the earlier $n=10$ example in Fig.~\figref{IcosaSlicesDifferent}(b).)
The trapezoids have the same dihedral angle and join the 
dome vertices of the blue face, so
all 
those dome vertices
must be at the same height from the base 
and
the odd 
base edges must have equal length, say $\ell$.
At the level of 
those dome vertices
we have an equiangular $4$-, $5$- or $6$-gon. 
The length of the edge parallel to base edge $e_i$ is
$|e_i| + \ell$.
%
\end{enumerate}
This completes the analysis of the edge-length constraints for $n=8,10,12$,
and so proves part (b) of Theorem~\thmref{EquiAng}.

\subsection{Not all domes are slices of general deltahedra}
\seclab{NotSlice}
From Section~\secref{n810Proof} above, we know that for $n=8$,
the even edges are those of a rectangle, $b,c,b,c$.
Furthermore, the construction of the dome is unique---in particular,
the faces on the odd edges must have downward normals and the faces on the even edges must have upward normals.
See Fig.~\figref{Octa_4x2}.

We now argue that the octagon domes where $b \neq c$ do not arise from slicing a generalized convex deltahedron.
If they did, it would be possible to construct domes on the top and the bottom of the octagon such that the two halves form a convex polyhedron. But on the bottom, by convexity, the faces on the odd edges would then have upward normals (after we flip it). Then the even edges would have to be the ones with downward normals so their lengths would need to be equal.
So the construction is not realizable if $b \neq c$.

A similar argument can be made for $n=12$.

\section{Further Results and Open Questions}
\label{sec:further-results}


We have characterized equiangular domeable polygons.
When we drop the equiangular condition,
many open problems remain. Here are a few.
\begin{enumerate}[(1)]
\item Is there any convex $n$-gon with $n > 12$ that can be domed?
A rough bound is $n\le 55$: every dome vertex has curvature at least $\frac{\pi}{3}$ so there are at most 11 dome vertices; every base vertex must be adjacent to a dome vertex and dome vertices have degree at most~$5$.
We can prove the stronger bound $n \le 24$.
%
\item Is there any convex $7$-gon that can be domed?
Fig.~\figref{9-gon_3Views} shows a ``one-story'' doming of a 
(non-equiangular) $9$-gon.
An $11$-gon can be constructed similarly, 
but the same construction scheme fails for a $7$-gon.
We do not know if there is a domeable $7$-gon via some other construction.
%
\item 
Is there any non-equilateral triangle that can be domed?
Glazarin and Pak conjectured~\cite{glazyrin2022domes} that,
even under their looser conditions,  
an isosceles triangle with edge lengths $2,2,1$
cannot be spanned.
\end{enumerate}
\begin{figure}[htbp]
\centering
\includegraphics[width=0.8\textwidth]{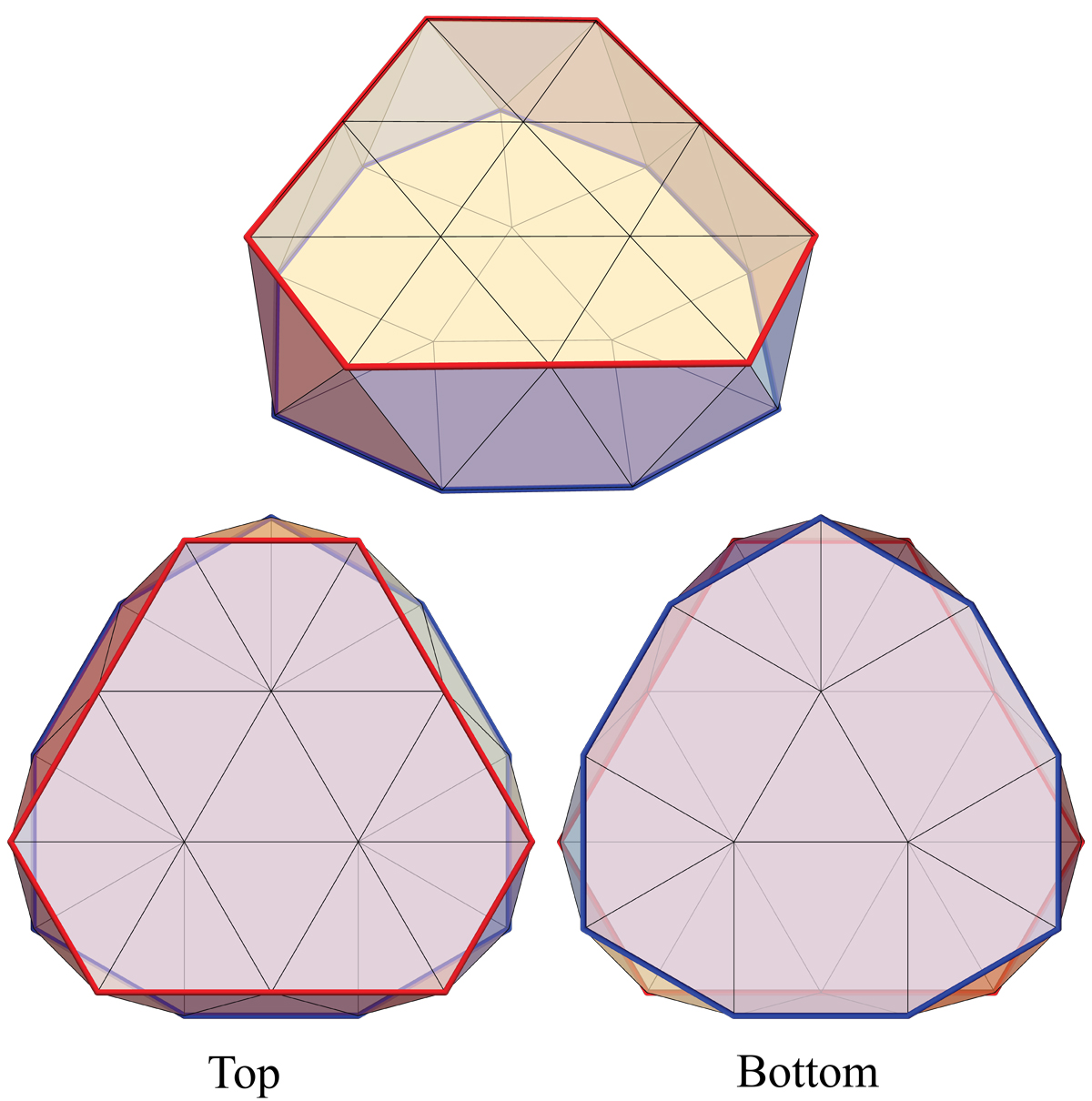}
\caption{The top (red) is a polyiamond of $13$ equilateral triangles.
The base (blue) is a $9$-gon with base angles $120^\circ$ and $150^\circ$.
}
\figlab{9-gon_3Views}
\end{figure}

\subparagraph*{Acknowledgments.}
We benefited from suggestions from six anonymous reviewers
and from comments at the EuroCG presentation~\cite{DomesEuroCG}.

\bibliography{refs}

\end{document}